\documentclass[10pt]{amsart}
\usepackage[english]{babel}
\usepackage[pdftex]{color,graphicx}

\newtheorem{theorem}{\textbf{Theorem}}

\newtheorem{lemma}{\textbf{Lemma}}

\newtheorem{passo}{\textbf{Step}}

\usepackage{mathrsfs}

\def\N {\mathbb{N}}
\def\Z {\mathbb{Z}}

\usepackage{amsfonts}
\usepackage{amssymb}

\theoremstyle{remark}

\numberwithin{equation}{section}

\begin{document}

	\title[Jacobi Symbol Criterion]{A Jacobi Symbol Criterion Involving $k$-Fibonacci and $k$-Lucas numbers and Integer Points on Elliptic Curves}
	
		\author{Gersica Freitas}
	\address{Campus de Engenharia e Ciências Agrárias, Universidade Federal de Alagoas, Rio Largo - AL, Brazil}
	\email{gersica.freitas@ceca.ufal.br}
	
	\author{Jean Lelis}
	\address{Faculdade de Matemática/ICEN/UFPA, Belém - PA, Brazil.}
	\email{jeanlelis@ufpa.br}
	
	\author{Elaine Silva}
	\address{Instituto de Matemática, Universidade Federal de Alagoas, Maceió - AL, Brazil}
	\email{elaine.silva@im.ufal.br}

	
	
	\begin{abstract}
		In 1989, Ming Luo \cite{L2} showed that the Fibonacci number $U_n$ is Triangular if and only if $n=\pm1,2,4,8,10$. For this, he established a Jacobi Symbol Criterion. Moreover, he observed that this problem is equivalent to finding all integer points on two elliptic curves. In this paper, we prove a Jacobi Symbol Criterion for more general families of binary recurrences. In addition, applying the criterion and elementary methods, we  determine all integer points on the elliptic curves $y^2=5x^2(x+3)^2+4(-1)^n$.
	\end{abstract}
	
	\maketitle
	
	\section{Introduction} \label{sec1}
	
	Let $U=(U_n)_{n\geq0}$ be a binary recurrence sequence defined by initial terms $U_0,U_1\in\Z$ and the recurrence relation
	\[
	U_{n+2}=AU_{n+1}+BU_{n}\qquad(n\geq0),
	\]
	where $A$ and $B$ are non-zero integer numbers. Let $\alpha=(A+\sqrt{D})/2$ and $\beta=(A-\sqrt{D})/2$ be the zeros of the characteristic polynomial of $U$ given by $p(x)=x^2-Ax-B$, where $D=A^2+4B$ is the discriminant of $U$, and
	\[
	a=U_1-\beta U_0,\quad b=U_1-\alpha U_0,\quad C=U_1^2-AU_0U_1-BU_0^2.
	\]
	The sequence $U$ is called non-degenerate if $C\neq0$ and $\alpha/\beta$ is not a root of unity. It is well-known that if $U$ is non-degenerate, then we have that 
	\begin{equation}\label{eq001}
		U_n=\frac{a\alpha^n-b\beta^n}{\alpha-\beta}
	\end{equation}
	for all integer number $n\geq 0$. In fact, \eqref{eq001} holds whenever $\alpha-\beta\neq0$.  
	
	From this point on, we assume that $B=1$, $A\geq1$ and $U$ is non-degenerate. Therefore, it is also well-known that $U$ has a so-called associate sequence $V=(V_n)_{n\geq0}$ for which 
	\[
	V_n^2-DU_n^2=4C(-1)^n
	\]
	holds for all $n\geq0$, where $V_0=2U_1-AU_0$, $V_1=AU_1+2BU_0$ and $V$ satisfies the same recurrence relation of $U$ (for more details see \cite{T1}).
	
	Observe that by our assumption that $B=1$, if we consider $U_0=0$, $U_1=1$ and $A=1$, then we obtain that $U$ and $V$ are the Fibonacci and Lucas sequences, respectively. These sequences are famous for having several identities and interesting properties associated with them. For these reasons, it is common to find several generalizations of these sequences. Many authors consider the problem of studying the binary recurrence, so-called $k$-Fibonacci sequence $F_k=(F_{k,n})_{n\geq0}$ given by $F_{k,0}=0$, $F_{k,1}=1$ and
	\[
	F_{k,n+2}=kF_{k,n+1}+F_{k,n}
	\] 
	for $n\geq0$, and its associated $k$-Lucas sequence $L_k=(L_{k,n})_{n\geq0}$, where $L_{k,0}=2$ and $L_{k,1}=k$ (for more details see \cite{B1,C1,F1,F2}). 
	
	There are many articles concerning the mixed exponential-polynomial Diophantine equation
	\[
	\qquad U_n=P(x),
	\] 
	where $P\in\mathbb{Z}[x]$ is a polynomial. In particular, there is a special interest in the case that $P(x)$ has degree 2. Since $Y=\pm L_{k,n}$ and $X=\pm F_{k,n}$ are the complete set of solutions of the Pell equations 
	\begin{equation}
		Y^2-(k^2+4)X^2=4(-1)^n,
	\end{equation}  
	the condition $F_{k,n}=P(x)$ is equivalent to finding all integer solutions of the Diophantine equation $Y^2-(k^2+4)(P(x))^2=\pm 4$ (for more details see \cite{J1}). In particular, if $P(x)$ has degree 2, this problem is equivalent to find all integer points on these two elliptic curves.
	
	Thus, given integers $a,b$ and $c$ with $a\neq0$, we have that the solutions of equation $\pm F_{k,n}=a(z+b)(z+c)$ are the $X$-coordinates of integer points on the following elliptic curves
	\begin{equation}\label{eq2}
		y^2=a^2(k^2+4)(z+b)^2(z+c)^2+4(-1)^n.
	\end{equation}
	
	For some values of the parameters in the equation \eqref{eq2}, we can determine all integer points on the curves above estimating linear forms in elliptic logarithms. For example, the \texttt{IntegralQuarticPoints()} subroutine of Magma can be used to determine the solution of some equations of this type. However, in this paper, we follow the ideas of many authors (for example \cite{L2,M1,R1}) and consider the problem of determining these integer points with elementary methods. Suppose that
	\begin{equation}\label{eq1}
		\pm F_{k,n}=a(x+b)(x+c)
	\end{equation}
	with $a,b,c$ integers with $a\neq0$. 
	
	Note that the equations \eqref{eq1} have solution if and only if $(\pm 4aF_{k,n}+\Delta)$ is a perfect square, where $\Delta=a^2(b+c)^2-4a^2bc=a^2(b-c)^2=d^2$. In this case, Jacobi symbol
	\[
	\left(\frac{\pm4aF_{k,n}+d^2}{s}\right)=1
	\]
	is valid for all odd positive integers $s$. The Jacobi symbol is an important tool in the study of Diophantine equations involving perfect squares (for more examples see \cite{K1,K2,L2,M1}). In order to determine an $s$ such that the Jacobi symbol above is $-1$, which implies that $(\pm 4aF_{k,n}+\Delta)$ is not a perfect square, we shall prove the following Jacobi Symbol Criterion, which we believe to have independent interest.
	
	\begin{theorem}\label{theo1}
		Let $a,d,k$ be positive integers, where $d$ and $k$ are odd with $d^2>8a$. If  $n\equiv \pm 2 \pmod{6}$ and $\gcd(a,L_n)=1$, then
		\[
		\left(\frac{\pm 4aF_{k,2n}+d^2}{L_{k,2n}}\right)=-\left(\frac{\pm 8aF_{k,n}+d^2L_{k,n}}{64a^2+(k^2+4)d^4}\right),
		\]
		whenever the right Jacobi symbol is proper. 
	\end{theorem}



{\section{Proof of Jacobi Symbol Criterion} \label{seccriterio}}

Firstly, let  $(L_{k,n})_{n\geq0}$ be the associate sequence given by $L_{k,n+2}=kL_{k,n+1}+L_{k,n},$ for $n\geq0$ and initial terms $L_{k,0}=2$ and $L_{k,1}=k$. The proof of Theorem \ref{theo1} requires the following identities:

\begin{eqnarray}
	L_{k,2n}&=&L_{k,n}^2-2(-1)^n;\label{1}\\
	F_{k,2n}&=&F_{k,n}L_{k,n};\label{2}\\
	2L_{k,2n}&=&(k^2+4)F_{k,n}^2+L_{k,n}^2.\label{3}
\end{eqnarray}

These identities are the generalization of well-known identities associated with Fibonacci and Lucas sequences, for more details see \cite{F1}. Since $k$ is odd and  $n\equiv \pm 2 \pmod{6}$, we have that $L_{k,n}\equiv 3 \pmod{4}$ and $L_{k,2n}\equiv 7 \pmod{8}$. So, we can consider the Jacobi symbol
\[
\left(\frac{\pm 4aF_{k,2n}+d^2}{L_{k,2n}}\right).
\]

Moreover, the Jacobi symbol $(2 \mid  L_{k,2n})=1$, thus 
\[
\left(\frac{\pm 4aF_{k,2n}+d^2}{L_{k,2n}}\right) =\left(\frac{2 }{L_{k,2n}}\right)\left(\frac{\pm 4aF_{k,2n}+d^2 }{L_{k,2n}}\right)=\left(\frac{\pm 8aF_{k,2n}+2d^2 }{L_{k,2n}}\right).
\]
By \eqref{1}, we have $2\equiv L_{k,n}^2 \pmod{L_{k,2n}}$, since $n$ is an even integer. Further, using \eqref{2} we obtain
\[
\left(\frac{\pm 4aF_{k,2n}+d^2}{L_{k,2n}}\right)=\left(\frac{\pm 8aF_{k,n}L_{k,n}+d^2L_{k,n}^2}{L_{k,2n}}\right).
\]
Note that $8a<d^2$ by hypothesis, thus $\pm 8aF_{k,n}L_{k,n}+d^2L_{k,n}^2>0$. So, by quadratic reciprocity it follows that
\[
\left(\frac{\pm 4aF_{k,2n}+d^2}{L_{k,2n}}\right)  =\left(\frac{L_{k,2n}}{\pm8aF_{k,n}L_{k,n}+d^2L_{k,n}^2}\right)
=\left(\frac{L_{k,2n}}{L_{k,n}}\right)\left(\frac{L_{k,2n}}{\pm 8aF_{k,n}+d^2L_{k,n}}\right),\]
since $d^2L^2_{k,n}\equiv 1 \pmod{4}$. Now, by \eqref{1} and using $L_{k,n}\equiv3 \pmod{4}$ we obtain $(L_{k,2n}\mid L_{k,n})=-(2\mid L_{k,n})$. Furthermore, \eqref{3} give us
\[
\left(\frac{\pm 4aF_{k,2n}+d^2}{L_{k,2n}}\right)=-\left(\frac{2}{L_{k,n}}\right)\left(\frac{\frac{1}{2}((k^2+4)F^2_{k,n}+L^2_{k,n})}{\pm 8aF_{k,n}+d^2L_{k,n}}\right).
\]

In order to exchange the sum $((k^2+4)F^2_{k,n}+L^2_{k,n})/2$ for a product, we can multiply it by a Jacobi symbol of a suitable perfect square. For this, we use the following identity
\[
\frac{16a^2d^2((k^2+4)F^2_{k,n}+L^2_{k,n})}{2}=(\pm 8aF_{k,n}+d^2L_{k,n})Q \mp (64a^3+(k^2+4)ad^4)F_{k,n}L_{k,n},
\]
where we consider $Q=(k^2+4)ad^2F_{k,n}+8a^2L_{k,n}$ in the case $8aF_{k,n}+d^2L_{k,n}$ and $Q=-(k^2+4)ad^2F_{k,n}+8a^2L_{k,n}$ in other case.

It follows that
\begin{eqnarray*}
	\left(\frac{\pm 4aF_{k,2n}+d^2}{L_{k,2n}}\right)=-\left(\frac{2}{L_{k,n}}\right)\left(\frac{\mp(64a^3+(k^2+4)ad^4)}{\pm 8aF_{k,n}+d^2L_{k,n}}\right)\left(\frac{F_{k,n}L_{k,n}}{\pm 8aF_{k,n}+d^2L_{k,n}}\right).
\end{eqnarray*}

Now,
\begin{eqnarray*}
   \left(\frac{L_{k,n}}{\pm 8aF_{k,n}+d^2L_{k,n}}\right)=-\left(\frac{\pm 8aF_{k,n}+d^2L_{k,n}}{L_{k,n}}\right)=\mp\left(\frac{2a}{L_{k,n}}\right)\left(\frac{F_{k,n}}{L_{k,n}}\right),
\end{eqnarray*}
and $n\equiv \pm2 \pmod{6}$ 
implies $F_{k,n}\equiv (\pm 1)(-1)^{\frac{k-1}{2}} \pmod{4}$, thus
\[
\left(\frac{F_{k,n}}{\pm 8aF_{k,n}+d^2L_{k,n}}\right)=(\pm 1)(-1)^{\frac{k-1}{2}}\left(\frac{\pm 8aF_{k,n}+d^2L_{k,n}}{F_{k,n}}\right)=\left(\frac{F_{k,n}}{L_{k,n}}\right).
\]

From this,
\begin{eqnarray*}
	\left(\frac{\pm 4aF_{k,2n}+d^2}{L_{k,2n}}\right)&=&\pm\left(\frac{2}{L_{k,n}}\right)\left(\frac{\mp(64a^3+(k^2+4)ad^4)}{\pm 8aF_{k,n}+d^2L_{k,n}}\right)\left(\frac{2a}{L_{k,n}}\right)\\
	&=&-\left(\frac{a}{L_{k,n}}\right)\left(\frac{a}{\pm 8aF_{k,n}+d^2L_{k,n}}\right)\left(\frac{64a^2+(k^2+4)d^4}{\pm 8aF_{k,n}+d^2L_{k,n}}\right)\\
	&=&-\left(\frac{a}{\pm 8aL_{k,n}F_{k,n}+d^2L_{k,n}^2}\right)\left(\frac{64a^2+(k^2+4)d^4}{\pm 8aF_{k,n}+d^2L_{k,n}}\right).
\end{eqnarray*}

Writing $a=2^sb$, with $2\nmid b$, we obtain
\begin{eqnarray*}
	\left(\frac{a}{\pm 8aL_{k,n}F_{k,n}+d^2L_{k,n}^2}\right)&=& \left(\frac{2^sb}{\pm 8aL_{k,n}F_{k,n}+d^2L_{k,n}^2}\right)\\
	&=&\left(\frac{2^s}{\pm 8aF_{k,n}L_{k,n}+d^2L^2_{k,n}}\right)\left(\frac{d^2L^2_{k,n}}{b}\right)\\
	&=&1,
\end{eqnarray*}
since
$\pm 8aF_{k,n}L_{k,n}+d^2L^2_{k,n} \equiv 1 \pmod 8$ and $b\mid a$. Finally, we conclude that 
\[
\left(\frac{\pm 4aF_{k,2n}+d^2}{L_{k,2n}}\right)=-\left(\frac{64a^2+(k^2+4)d^4}{\pm 8aF_{k,n}+d^2L_{k,n}}\right)=-\left(\frac{\pm 8aF_{k,n}+d^2L_{k,n}}{64a^2+(k^2+4)d^4}\right).
\]
\qed

\section{The Curves $y^2=5x^2(x+3)^2+4(-1)^n$}\label{section 3}

In this section, we consider the elliptic curves $y^2=5x^2(x+3)^2+4(-1)^n$ to exemplify the method. It is well-known that $X=\pm F_n$ and $Y=\pm L_n$ are the complete set of solutions of the Diophantine equations $$Y^2-5X^2= 4(-1)^n.$$ So, we conclude that
the curves $y^2=5x^2(x+3)^2+4(-1)^n$ have integer points if and only if the equation $\pm F_n=x(x+3)$ has a solution. We shall prove the following theorem 

\begin{theorem}\label{theo2}
If $(x,y)$ is a integer point on the elliptic curves $$y^2=5x^2(x+3)^2+4(-1)^n$$ then $(x,y,n)\in\{(-3,\pm 2,0),(-2,\pm4,3),(-1,\pm4,3),(0,\pm 2,0)\}$.
\end{theorem}

Clearly, $-F_n=x(x+3)$ if and only if $x\in\{-3,-2,-1,0\}$. Further, if $n=0$ and $x=0$, then $F_n=x(x+3)$. We shall prove that there are no other solutions. To do this, we will use Jacobi Symbol Criterion and some auxiliary lemmas. The following lemma will be used to reduce $F_{n}$ modulo $L_{2k}$, in order to apply Jacobi Symbol Criterion \ref{theo1}. 


\begin{lemma} \label{lema1}
	For all integers $k$ and $m$, and $g$ odd,
	\[
	F_{2kg+m}\equiv \left\{
	\begin{array}{ccc}
		F_{2k+m}\pmod{L_{2k}} & if & g\equiv 1 \pmod{4}\\
		-F_{2k+m}\pmod{L_{2k}} & if & g\equiv 3 \pmod{4}.\\
	\end{array}
	\right.
	\]
\end{lemma}

The proof of this lemma can be found in \cite{M1}. On the way to prove Theorem \ref{theo2} we use that $n=2\cdot 2^w\cdot5^2\cdot 7t$, with $w\geq3$ and $t$ odd. In order to guarantee this condition we proof the following lemma.

\begin{lemma}\label{lema2}
	If $F_n=x(x+3)$, for $n,x\in\N$, then $n\equiv 0 \pmod{2800}$.
\end{lemma}

\begin{proof}
	If $n$ and $x$ are integer numbers such that $F_n=x(x+3)$, then $4F_n+9$ is a perfect square. Hence, the Jacobi symbol $(4 F_n+9\mid Q)=1$ for every odd positive integer $Q$. So the idea of the proof is to show that if $n$ is not congruent to $0$ modulo $2800$, then there is an odd prime $Q$ such that $(4 F_n+9\mid Q)=-1$. 
	
	The proof will be divided into five steps. In each step, we use some Jacobi symbol properties and that $(F_n)_{n\geq0}$ is periodic modulo $Q$.
	
	\begin{passo}
		$n\equiv 0 \pmod{4}$.
	\end{passo}
	\noindent In fact, $F_n$ is periodic modulo $3$ and $7$ with period $8$ and $16$, respectively, and
	\begin{eqnarray*}
			\left(\frac{4 F_n+9}{3}\right)=-1,& &\mbox{if}\ n\equiv 3,5,6 \pmod{8};\\
			\left(\frac{4 F_n+9}{7}\right)=-1,& &\mbox{if}\ n\equiv \pm 1, \pm 2, \pm 3, \pm 6, \pm 7 \pmod{16}.
	\end{eqnarray*}
	thus, if $(4 F_n+9\mid Q)=1$ for all odd positive integers $Q$, then $n \equiv 0 \pmod{4}$.
	
	\begin{passo}
		$n\equiv 0 \pmod{20}$.
	\end{passo}
	\noindent By step 1, we have that $n\equiv 0,4,8,12,16\pmod{20}$ and $n\equiv 0,4,8,2,6\pmod{10}$. Using that $F_n$ is periodic modulo $5$ and modulo $11$, with periods $20$ and $10$, respectively, and
	\[
	\left(\frac{4 F_n+9}{Q}\right)=-1,\ \mbox{ if }\ (n,Q)=(4,11),(8,5),(12,11), (16,5),
	\]
	we get $n\equiv 0 \pmod{20}$.
	\begin{passo}
		$n\equiv 0 \pmod{100}$.
	\end{passo}
	
	\noindent Since $n\equiv 0 \pmod{20}$, we have $n\equiv 0,\pm20,\pm40,\pm60,\pm80,100 \pmod{200}$. Note that $F_n$ is periodic modulo $401$, with periods $200$ and
	\[
	\left(\frac{4 F_n+9}{401}\right)=-1,\ \mbox{ if }\ n\equiv \pm20,40,60,\pm80 \pmod{200},
	\]
	 moreover, if $n\equiv 140\pmod{200}$, then $n\equiv 40\pmod{100}$ and if $n\equiv 160\pmod{200}$, then $n\equiv 10\pmod{50}$. Using that $F_n$ is periodic modulo $3001$ and $101$, with periods $100$ and $50$, respectively, and  
	\[
	\left(\frac{4 F_n+9}{Q}\right)=-1,\ \mbox{ if }\ (n,Q)=(40,3001),(10,101),
	\]
	we obtain that $n\equiv 0\pmod{100}$.
	\begin{passo}
		$n\equiv 0 \pmod{700}$.
	\end{passo}
	\noindent Firstly, $F_n$ is periodic modulo $13$ and $29$, with periods $28$ and $14$, respectively. Since $100k\equiv0,16,4,20,8,24,12\pmod{28}$ and $100k\equiv0,2,4,8,10,12\pmod{14}$ for $k=0,1,2,3,4,5,6$, we have that $\left(4 F_n+9 \mid Q\right)=-1$ if $(n,Q)=(16,13),$ $(4,13),$ $(6,29),$ $(8,13),$ $(10,29)$. Thus, either $n\equiv 0 \pmod{700}$ as we claim or $n\equiv 600\pmod{700}$. 
	
	\noindent Note that $F_n$ is periodic modulo $281$ and $2801$ with periods $56$ and $1400$, respectively, and $n\equiv40,12\pmod{56}$ and $n\equiv600,1300\pmod{1400}$ in the case $n\equiv 600\pmod{700}$, but we have that $\left(4 F_n+9 \mid Q\right)=-1$ if  $(n,Q)=(600,2801),(12,281)$, so we conclude that $n\equiv 0\pmod{700}$.

	\begin{passo}
		$n\equiv 0 \pmod{2800}$.
	\end{passo}
	\noindent Finally, as $n\equiv 0 \pmod{700}$ and $F_n$ is periodic modulo $47$ and $1601$ with periods $32$ and $160$, respectively, we consider that $700k\equiv 0,28,24,20,16,12,8,4\pmod{32}$ and $700k\equiv 0,60,120,20,80,140,40,100\pmod{160}$ for $k=0,1,2,3,4,5,6,7$. Since the Jacobi symbol $\left(4 F_n+9 \mid Q\right)=-1$, if  $(n,Q)=(28,47),$ $(24,47),$ $(20,47),$ $(140,1601),$ $(8,47),$ $(100,1601)$ we conclude that $k=0,4$, hence we conclude that $n\equiv 0\pmod{2800}$ and this complete the proof.
\end{proof}

For the purpose of simplifying the Jacobi Symbol Criterion proved in Theorem \ref{theo1}, we shall prove the following lemma. Before that, let  $\nu_p(r)$ be the $p$-adic valuation of a integer number $r$, i. e., the exponent of the highest power of a prime $p$ which divides $r$. The $p$-adic valuation of a Fibonacci number was completely characterized, see \cite{L1}. For instance, if $m\equiv 0 \pmod{8}$ we have $\nu_7(F_m)=\nu_7(m)+1$, now we have conditions to prove the following lemma.

\begin{lemma}\label{lema3}
	If $m\equiv 0 \pmod{16}$, then $\left(\pm 8 F_m+9L_m\mid7\right)=1$.
\end{lemma}

\begin{proof}
	Note that $7\mid F_{m/2}$, since $m/2\equiv0\pmod8$ and $\nu_7(F_{m/2})=\nu_7(m/2)+1\geq 1$. Thus, by $(2.3)$,
	\[ \left( \frac{L_m}{7}\right)=\left(\frac{\frac{1}{2}(5F_{m/2}^2+L_{m/2}^2)}{7} \right)=\left(\frac{10F_{m/2}^2+2L_{m/2}^2}{7}\right)=\left(\frac{2L_{m/2}^2}{7}\right)=1,\]
	hence
	\[
	\left(\frac{\pm 8F_m+9L_m}{7}\right)=\left(\frac{9L_m}{7}\right)=\left(\frac{L_m}{7}\right)=1.
	\]
\end{proof}

Note that, by the lemma above and the Jacobi Symbol Criterion with $a=1$, $k=1$, $d=3$ and $n\equiv\pm2\pmod 6$. If $n\equiv 0 \pmod{16}$, then we have that
\begin{eqnarray*}
\left(\frac{\pm 4F_{2n}+9}{L_{2n}}\right)&=&-\left(\frac{\pm 8F_{n}+9L_{n}}{64+5\cdot81}\right)=-\left(\frac{\pm 8F_{n}+9L_{n}}{469}\right)\\                                             
                                                &=&-\left(\frac{\pm 8F_{n}+9L_{n}}{7}\right)\left(\frac{\pm 8F_{n}+9L_{n}}{67}\right)\\
                                                &=&-\left(\frac{\pm 8F_{n}+9L_{n}}{67}\right).
\end{eqnarray*} 
The following lemma will be useful to calculate $\left(\pm 8 F_n+9L_n \mid 67\right)$ and we use in the proof the Sage Mathematics Software System \cite{S1}.

\begin{lemma}\label{lema4} Let $w\geq3$ be a positive integer number. If $w\equiv 0,1,2,3,4,5,6,7\pmod{8}$, then we have 
	
	\begin{eqnarray*}
		\begin{cases}
			F_{2^w}\equiv18,62,64,21,49,5,3,46\pmod{67};\\
			L_{2^w}\equiv63,14,60,47, 63,14,60,47\pmod{67};\\
			F_{2^w\cdot7}\equiv4,65,37,10,63,2,30,57\pmod{67};\\
			L_{2^w\cdot7}\equiv33,15,22,13,33,15,22,13\pmod{67};\\
			F_{2^w\cdot5^2}\equiv21,49,5,3,46,18,62,64\pmod{67};\\
			L_{2^w\cdot5^2}\equiv47,63,14,60, 47,63,14,60\pmod{67};\\
			F_{2^w\cdot5^2\cdot7}\equiv10,63,2,30,57,4,65,37\pmod{67};\\
			L_{2^w\cdot5^2\cdot7}\equiv13,33,15,22,13,33,15,22\pmod{67},
		\end{cases}
	\end{eqnarray*}
	respectively.
\end{lemma}

\begin{proof}
	Note that $F_n$ and $L_n$ are periodic modulo $67$ with period $136$. Let $n_w=2^wt$ be a positive integer, where $w\geq 3$ and $t$ is a fixed integer with $\gcd(t,17)=1$. Since $(n_w/2^3)_{w\geq 3}$ is a sequence of integer numbers periodic modulo 17 with period $8$, we have that $F_{n_w}$ and $L_{n_w}$ are periodic modulo $67$ with period 8. After some calculations, we obtain the values in the lemma and this complete the proof.
\end{proof}


\noindent \textit{Proof of Theorem 2:} Now, we are able to study the equation $y^2=5x^2(x+3)^2\pm4$, we shall prove that the only integer points on these elliptic curves are $(x,y)=(0,\pm 2)$. To obtain a contradiction, suppose that $F_n=x(x+3)$ has a solution with $(n,x)\neq (0,0)$. By Lemma \ref{lema2}, we have that $n=2\cdot 2^w\cdot5^2\cdot 7t$, with $w\geq3$ and $t$ odd. Moreover, we can write $n=2kg$ such that $3\nmid k$ and $g$ is odd. Hence, by Lemma \ref{lema1}, we have
\[
\left(\frac{ 4F_n+9}{L_{2k}}\right)=\left\{
\begin{array}{ccc}
\left(\frac{ 4F_{2k}+9}{L_{2k}}\right), & \text{if} & g\equiv 1\pmod{4},\\
\left(\frac{- 4F_{2k}+9}{L_{2k}}\right), & \text{if} & g\equiv 3\pmod{4}.\\
\end{array}
\right.
\]

\noindent\textbf{Case 1:} $t\equiv 1\pmod{4}$.
\begin{itemize}
\item If $w\equiv 0,3,5,6,7 \pmod{8}$, then we consider $k=2^w$ and $g=5^2\cdot7t$. Note that $g\equiv3\pmod4$, thus by Lemma \ref{lema1}, the Criterion and Lemma \ref{lema4}, we obtain
\[
\left(\frac{4F_n+9}{L_{2k}}\right)=\left(\frac{-4F_{2k}+9}{L_{2k}}\right)=-\left(\frac{-8F_k+9L_k}{67}\right)=-1.
\]
\item If $w\equiv 1,2,4\pmod{8}$, then we take $k=2^w\cdot5^2\cdot7$ and $g=t$. Using again Lemma \ref{lema1}, the Criterion and Lemma \ref{lema4}, we have
\[
\left(\frac{4F_n+9}{L_{2k}}\right)=\left(\frac{4F_{2k}+9}{L_{2k}}\right)=-\left(\frac{8F_k+9L_k}{67}\right)=-1.
\]
\end{itemize}

\noindent\textbf{Case 2:} $t\equiv 3\pmod{4}$.
\begin{itemize}
\item If $w\equiv 1,2,3,4,7 \pmod{8}$, then we take $k=2^w$ and $g=5^3\cdot7t$. Note that $g\equiv1\pmod4$, so in this case we have
\[
\left(\frac{4F_n+9}{L_{2k}}\right)=\left(\frac{4F_{2k}+9}{L_{2k}}\right)=-\left(\frac{8F_k+9L_k}{67}\right)=-1.
\]

\item Finally, if $w\equiv 0,5,6 \pmod{8}$, then putting $k=2^w\cdot5^2\cdot7$ and $g=t$, we conclude that
\[
\left(\frac{4F_n+9}{L_{2k}}\right)=\left(\frac{-4F_{2k}+9}{L_{2k}}\right)=-\left(\frac{-8F_k+9L_k}{67}\right)=-1.
\]
\end{itemize}

Therefore, we obtain a contradiction in all cases. These contradictions occur by supposing that $F_n=x(x+3)$ has a solution other than $(n,x)=(0,0)$. Hence, if $(x,y)$ is a integer point on the elliptic curves $$y^2=5x^2(x+3)^2+4(-1)^n$$ then $(x,y,n)\in\{(-3,\pm 2,0),(-2,\pm4,3),(-1,\pm4,3),(0,\pm 2,0)\}$, and this complete the proof of Theorem \ref{theo2}.


\section*{Acknowledgement}
The authors are grateful to
the referee for his/her valuable suggestions which helped to improve the quality of this
work. They also thank Diego Marques and Pavel Trojovský for initial discussions on the subject. The third author was partially supported by the National Council for
Scientific and Technological Development – CNPq, Brazil (Grant 409198/2021-8).

\end{document}